\documentclass[a4paper,12pt]{amsart}
\usepackage{amssymb,amscd,amsmath}

\theoremstyle{plain}
\newtheorem{thm}{Theorem}[section]
\newtheorem{prop}[thm]{Proposition}
\newtheorem{lem}[thm]{Lemma}
\newtheorem{cor}[thm]{Corollary}

\theoremstyle{definition}
\newtheorem{defn}[thm]{Definition}

\theoremstyle{remark}
\newtheorem{rem}{Remark}[thm]

\newcommand{\abs}[1]{\left\vert{#1}\right\vert}
\newcommand{\set}[1]{\left\{{#1}\right\}}

\newcommand{\Natural}{\mathbb{N}}

\newcommand{\El}{\ell}
\newcommand{\Pl}{{p^\El}}

\newcommand{\Fq}{{\FF_q}}

\newcommand{\FF}{\mathbb{F}}

\newcommand{\FE}{\mathbb{E}}
\newcommand{\FK}{\mathbb{K}}
\newcommand{\Fk}{\Bbbk}

\newcommand{\dv}{\!\mid\!}
\newcommand{\ndv}{\!\nmid\!}

\newcommand{\gn}[1]{{\langle{#1}\rangle}}

\newcommand{\Dim}[1]{{\mathrm{dim}\,{#1}}}

\newcommand{\DimK}[1]{{\mathrm{dim}_\FK{#1}}}
\newcommand{\Dimk}[1]{{\mathrm{dim}_\Fk{#1}}}

\newcommand{\slnq}{\mathfrak{sl}(n,q)}
\newcommand{\glnq}{\mathfrak{gl}(n,q)}

\newcommand{\glnK}{\mathfrak{gl}(n,\FK)}
\newcommand{\slnK}{\mathfrak{sl}(n,\FK)}
\newcommand{\glnk}{\mathfrak{gl}(n,\Fk)}

\newcommand{\GLnq}{\mathrm{GL}(n,q)}

\newcommand{\GLnK}{\mathrm{GL}(n,\FK)}

\newcommand{\GLnk}{\mathrm{GL}(n,\Fk)}

\begin{document}

\title{Cartan subalgebras in general and special linear algebras}

\author{Kyoung-Tark Kim}
\address{Department of mathematics, Pusan National University, jang-jeon dong, Busan, Republic of Korea}
\email{poisonn00@hanmail.net}

\maketitle

\begin{abstract}
  In this paper we study Cartan subalgebras in general and special linear algebras over a field of positive characteristic.
  We determine the conjugacy classes of Cartan subalgebras under the general linear groups, and count the explicit number of all Cartan subalgebras from its conjugacy when the base field is an arbitrary finite field.
\end{abstract}

\section{Introduction}

\noindent Let $\FK$ be a field of positive characteristic $p$.
By the \textit{general linear algebra} $\glnK$ we mean the vector space of all $n\times n$ matrices over $\FK$ together with a unary operation $X \mapsto X^p$, called the \textit{$p$-map}, and a binary operation $(X,Y) \mapsto [X,Y] = XY - YX$, called the \textit{commutator} or the \textit{Lie bracket} where $X$ and $Y$ are in $\glnK$.
The \textit{special linear algebra} $\slnK$ is the subspace of all trace zero matrices in $\glnK$ together with the inherited operations.
(It is well-known and easy to see that $\slnK$ is closed under the $p$-map and the commutator.)

A \textit{torus} $\mathfrak{T}$ is defined as a subspace of $\glnK$ satisfying
\begin{enumerate}
\item [(T1)] $\mathfrak{T}$ is closed under the $p$-map;
\item [(T2)] $[X,Y] = 0$ for all $X, Y \in \mathfrak{T}$;
\item [(T3)] each matrix in $\mathfrak{T}$ is \textit{absolutely semisimple}, i.e., diagonalizable over an algebraic closure $\Fk$ of $\FK$.
\end{enumerate}

Let $\mathfrak{D}$ and $\mathfrak{D}_0$ be respectively the subsets of all diagonal matrices in $\glnK$ and $\slnK$.
Clearly, $\mathfrak{D}$ and $\mathfrak{D}_0$ are tori with $\Dim{\mathfrak{D}} = n$ and $\Dim{\mathfrak{D}_0} = n-1$.
They are maximal by the following basic observation:

\begin{prop}\label{maximaldim}
  Let $\mathfrak{T}$ be a torus.
  Then, $\mathfrak{T}$ is maximal in $\glnK$ (in $\slnK$, respectively) if and only if $\Dim{\mathfrak{T}}$ is equal to $n$ (resp. $n-1$).
\end{prop}

\begin{proof}
  Since all matrices in $\mathfrak{T}$ commute there is $U \in \GLnk$ such that $U^{-1}\overline{\mathfrak{T}}U \subseteq \overline{\mathfrak{D}}$ where, for any $\FK$-vector space $\mathfrak{V}$, $\overline{\mathfrak{V}} := \Fk \otimes_\FK \mathfrak{V}$ denotes the usual scalar extension.
  Since $\Dimk{\overline{\mathfrak{D}}} = \DimK{\mathfrak{D}} = n$ and $\Dimk{U^{-1}\overline{\mathfrak{T}}U} = \DimK{\mathfrak{T}}$ we have $\DimK{\mathfrak{T}} \leq n$.
  So if $\DimK{\mathfrak{T}} = n$ then $\mathfrak{T}$ is maximal in $\glnK$.
  Conversely suppose that $\mathfrak{T}$ is maximal in $\glnK$.
  Since $\overline{\mathfrak{T}}$ (hence $U^{-1}\overline{\mathfrak{T}}U$) is a maximal torus in $\glnk$ by \cite[Theorem 2.15, p. 73]{Winter}, we have $\DimK{\mathfrak{T}} = n$.
  The proof in $\slnK$ is similar.
\end{proof}

In a Lie algebra $\mathfrak{g}$ of arbitrary characteristic one can define a \textit{Cartan subalgebra} (abbreviated, \textit{CSA}) $\mathfrak{H}$ as a Lie subalgebra of $\mathfrak{g}$ satisfying
\begin{enumerate}
\item [(C1)] $\mathfrak{H}$ is a nilpotent Lie algebra;
\item [(C2)] the normalizer of $\mathfrak{H}$ in $\mathfrak{g}$ is $\mathfrak{H}$ itself, i.e., $\mathfrak{H}$ is self-normalizing.
\end{enumerate}

Traditionally the role of CSAs is central in the theory of finite dimensional Lie algebras of characteristic \emph{zero}:
The definition of a CSA gives the \textit{Fitting-Zassenhaus} decomposition of a Lie algebra.
In the case of characteristic zero, the existence of a CSA was proved by showing that a CSA is the centralizer of a \textit{regular} element, namely, a minimal \textit{Engel} subalgebra.
Moreover, when the base field is algebraically closed, the well-known conjugacy theorem for CSAs assures the `uniqueness'.

In this paper we focus on CSAs in general and special linear algebras (which are probably the most naive to understand) over a field of positive characteristic, and consider the structure and the conjugacy problems about them.
In particular, when the ground field is finite, we concentrate on counting problem for CSAs.

In fact the algebraic structures like $\glnK$ or $\slnK$ have a given name - a \textit{restricted Lie algebras} or a \textit{Lie $p$-algebras} - initially  introduced by N. Jacobson.
In a restricted Lie algebra there is a fundamental connection between CSAs and maximal tori:

\begin{thm}\cite{Winter}\label{winterthm}
  Let $\mathfrak{g}$ be a restricted Lie algebra.
  Then, $\mathfrak{H}$ is a CSA of $\mathfrak{g}$ if and only if $\mathfrak{H}$ is the centralizer of a maximal torus in $\mathfrak{g}$.
\end{thm}

From the later sections of this paper we shall obtain the followings:

\begin{prop}\label{mtingl}
  Every maximal torus in $\glnK$ is self-centralizing.
\end{prop}

\begin{proof}
  It follows from Propositions \ref{maximaldim}, \ref{sscaa} and Lemma \ref{regimbedstd}.
\end{proof}

\begin{prop}\label{mtinsl}
  If $p = 2$ then $\mathfrak{sl}(2,\FK)$ is nilpotent.
  If $p > 2$ or $n > 2$ then every maximal torus in $\slnK$ is self-centralizing.
\end{prop}

\begin{proof}
  If $\mathrm{char}\,\FK = 2$ then a basis for $\mathfrak{sl}(2,\FK)$ is $\{ \left(\begin{smallmatrix} 1 & 0 \\ 0 & 1 \end{smallmatrix}\right), \left(\begin{smallmatrix} 0 & 1 \\ 0 & 0 \end{smallmatrix}\right), \left(\begin{smallmatrix} 0 & 0 \\ 1 & 0 \end{smallmatrix}\right) \}$.
  Since $[\mathfrak{sl}(2,\FK), \mathfrak{sl}(2,\FK)] = \FK \left(\begin{smallmatrix} 1 & 0 \\ 0 & 1 \end{smallmatrix}\right)$ and $[\mathfrak{sl}(2,\FK), \FK \left(\begin{smallmatrix} 1 & 0 \\ 0 & 1 \end{smallmatrix}\right)] = 0$, $\mathfrak{sl}(2,\FK)$ is nilpotent.
  The second assertion follows from Corollary \ref{slselfcentral}.
\end{proof}

Therefore the maximal tori in $\glnK$ or $\slnK$ are precisely the CSAs by Theorem \ref{winterthm} and Propositions \ref{mtingl}, \ref{mtinsl} unless $p = n = 2$, in which case, we know that $\mathfrak{sl}(2, \FK)$ is itself a CSA.

In section 2 we prepare and study basic structures of maximal tori in $\glnK$.
In section 3 we finally gain the number of all maximal tori in $\glnK$ when $\FK$ is an arbitrary finite field.
This result follows from a consideration about the conjugacy classes under $\GLnK$.
In section 4 we consider the case in $\slnK$.
In particular we also obtain the number of all maximal tori in $\slnK$ when $\FK$ is a finite field.

\section{Basic study in maximal tori in $\glnK$}

\noindent In this section we denote by $\mathfrak{M}$ a torus which is \emph{maximal} in $\glnK$.

\begin{prop}\label{closedmp}
  The following hold for a maximal torus in $\glnK$.
  \begin{enumerate}
  \item $\mathfrak{M}$ is closed under the matrix product;
  \item $I_n \in \mathfrak{M}$ where $I_n$ is the $n \times n$ identity matrix.
  \end{enumerate}
\end{prop}

\begin{proof}
  (i)
  Let $X$ and $Y$ in $\mathfrak{M}$.
  Since $XY = YX$, $X$ and $Y$ are simultaneously diagonalizable over $\Fk$.
  So $XY$ is absolutely semisimple.
  Since $XY$ commutes with all matrices in $\mathfrak{M}$ the maximality of $\mathfrak{M}$ implies $\mathfrak{M} + \sum_{i=0}^{\infty} \FK (XY)^{p^i} = \mathfrak{M}$, i.e., $XY \in \mathfrak{M}$.
  (ii)
  Since $I_n$ is diagonal and commutes with all matrices in $\mathfrak{M}$ the maximality yields $I_n \in \mathfrak{M}$.
\end{proof}

\begin{prop}\label{sscaa}
  $\mathfrak{M}$ is a semisimple commutative algebra over $\FK$.
\end{prop}

\begin{proof}
  By Proposition \ref{closedmp}, $\mathfrak{M}$ is a commutative associative $\FK$-algebra.
  Since $\mathfrak{M}$ has no nonzero nilpotent matrix $\mathfrak{M}$ is semisimple.
\end{proof}

\begin{cor}\label{dsofefs}
  We have $\mathfrak{M} \cong \FF_1 \oplus \cdots \oplus \FF_t$ as associative $\FK$-algebras where $\FF_i$ is a separable extension field of $\FK$ for each $i \in \{1,\ldots,t\}$.
\end{cor}

\begin{proof}
  By Proposition \ref{sscaa} and the Wedderburn structure theorem we know $\mathfrak{M} \cong \FF_1 \oplus \cdots \oplus \FF_t$ for some extension fields $\FF_1, \ldots, \FF_t$ of $\FK$.
  Since every element of $\mathfrak{M}$ is diagonalizable the minimal polynomial over $\FK$ of an element of $\FF_i$ must be separable, i.e., $\FF_i$ is separable over $\FK$.
\end{proof}

Let $\varphi:\FF_1 \oplus \cdots \oplus \FF_t \rightarrow \mathfrak{M}$ be an isomorphism.
Put $E_i := \varphi(1_{\FF_i})$ for each $i \in \{1,\ldots,t\}$.
Then $E_1, \ldots, E_t$ are orthogonal primitive idempotents with $I_n = \sum E_i$.
Since these are uniquely determined by $\mathfrak{M}$ so are $n_i := \Dim{\mathfrak{M}E_i} = \DimK{\FF_i}$.
We may assume that $n_1 \geq \cdots \geq n_t$.

\begin{defn}\label{mttype}
  The sequence $(n_1, \ldots, n_t)$ is called the \textit{type} of $\mathfrak{M}$.
  \begin{enumerate}
  \item $\mathfrak{M}$ is called \textit{division} if $\mathfrak{M}$ has type $(n)$.
  \item $\mathfrak{M}$ is called \textit{split} if $\mathfrak{M}$ has type $(1, \ldots, 1)$;
  \end{enumerate}
\end{defn}

\begin{defn}\label{equivmod}
  Let $A_1$ and $A_2$ be associative $\FK$-algebras.
  Let $M_1$ be an $A_1$-module and $M_2$ an $A_2$-module.
  We say that $M_1$ is \textit{equivalent} to $M_2$ if there exist an isomorphism $f : A_1 \rightarrow A_2$ and a $\FK$-linear bijection $g : M_1 \rightarrow M_2$ such that $g(am) = f(a)g(m)$ for all $a \in A_1$ and $m \in M_1$.
\end{defn}

Evidently, the regular modules for $\FF_1 \oplus \cdots \oplus \FF_t$ and $\mathfrak{M}$ are equivalent via the map $\varphi$.
Moreover we know the following fundamental fact:

\begin{lem}\label{regimbedstd}
  Let $\mathfrak{A}\subseteq \mathrm{Mat}(n,\FK)$ be a semisimple commutative algebra over $\FK$ with $I_n$.
  Then there is an $\mathfrak{A}$-module monomorphism $\mathfrak{A} \hookrightarrow \FK^n$.
\end{lem}

\begin{proof}
  Since $\mathfrak{A}$ is semisimple, $\mathfrak{A} = \mathfrak{I}_1 \oplus \cdots \oplus \mathfrak{I}_t$ where $\mathfrak{I}_i$'s are irreducible $\mathfrak{A}$-modules, i.e., minimal left ideals of $\mathfrak{A}$.
  Since $\mathfrak{A}$ is commutative, $\mathfrak{I}_i$'s are indeed two-sided ideals of $\mathfrak{A}$.
  Therefore if $I_n = A_1 + \cdots + A_t$ for some $A_i \in \mathfrak{I}_i$ then $A_1,\ldots,A_t$ are orthogonal primitive idempotents in $\mathfrak{A}$.
  Since $A_i \neq 0$ we can choose $v_i \in \FK^n$ such that $A_i v_i \neq 0$.
  If we define $\psi_i : \mathfrak{I}_i \rightarrow A_i \FK^n$ to be an $\mathfrak{A}$-module homomorphism generated by $A_i \mapsto A_iv_i$ then $\psi_i$ is injective whence so is $\psi_1 \oplus \cdots \oplus \psi_t : \mathfrak{A} \rightarrow \FK^n$.
\end{proof}

\begin{cor}\label{regstdmodofmt}
  The regular module for $\mathfrak{M}$ is isomorphic to the standard module $\FK^n$ for $\mathfrak{M}$ (defined from the inclusion $\mathfrak{M} \subseteq \mathrm{Mat}(n, \FK)$).
\end{cor}

\begin{proof}
  Since $\Dim{\mathfrak{M}} = \Dim{\FK^n}$ we are done by Lemma \ref{regimbedstd}.
\end{proof}

The regular module for $\FF_1 \oplus \cdots \oplus \FF_t$ gives rise to an equivalent module $\FK^n$:
If $\{ \xi^{(k)}_1 , \ldots, \xi^{(k)}_{n_k} \}$ is a $\FK$-basis of $\FF_k$ for each $k \in \{1,\ldots, t\}$ then we define a (faithful) regular representation $\rho_k : \FF_k \rightarrow \mathrm{Mat}(n_k, \FK)$ by $\rho_k(\eta) := (\alpha_{i,j}(\eta))$ where $\eta \xi^{(k)}_j = \sum_{i=1}^{n_k} \alpha_{i,j}(\eta) \xi^{(k)}_i$ for $j \in \{1,\ldots,n_k \}$.
Since $\rho_1(\FF_1) \oplus \cdots \oplus \rho_t(\FF_t)$ can be regarded as a direct sum in $\mathrm{Mat}(n, \FK)$ we obtain an equivalent module $\FK^n$ for $\rho_1(\FF_1) \oplus \cdots \oplus \rho_t(\FF_t)$ via the correspondence $\{\xi^{(k)}_j \mid k \in \{ 1, \ldots, t \},\, j \in \{1, \ldots, n_k \} \} \leftrightarrow \{ e_1, \ldots, e_n \}$ where $e_i \in \FK^n$ is the standard unit vector in $\FK^n$ for each $i \in \{ 1,\ldots,n\}$.

\begin{prop}\label{conjugatebd}
  We have $U^{-1}\mathfrak{M}U = \rho_1(\FF_1) \oplus \cdots \oplus \rho_t(\FF_t)$ for some $U \in \GLnK$ and regular representations $\rho_i : \FF_i \rightarrow \mathrm{Mat}(n_i,\FK)$.
\end{prop}

\begin{proof}
  From Corollary \ref{regstdmodofmt} and the preceding argument we conclude that the modules $\FK^n$ for $\mathfrak{M}$ and $\rho_1(\FF_1) \oplus \cdots \oplus \rho_t(\FF_t)$ are equivalent.
\end{proof}

\section{Maximal tori in $\glnK$ when $\FK$ is finite}

\noindent Our goal of this section is to obtain the number of all maximal tori (or Cartan subalgebras) in $\glnK$ when $\FK$ is a finite field with $\mathrm{char}\,\FK = p$.

Throughout this section we put $q := \Pl$ and assume $\FK = \Fq$, that is, a finite field with $\Pl$ elements.
We use the symbol $\mathfrak{M}$ as a maximal torus in $\glnK$.
Also we shall frequently see the following notations:
\begin{enumerate}
\item $\mathrm{Cl}(\mathfrak{M}) := \{ U^{-1}\mathfrak{M}U \mid U \in \GLnq \}$;
\item $\mathrm{N}(\mathfrak{M}) := \{ U \in \GLnq \mid U^{-1}\mathfrak{M}U = \mathfrak{M}\}$.
\end{enumerate}

\begin{prop}\label{numconjugacy}
  The type of $\mathfrak{M}$ determines the conjugacy class $\mathrm{Cl}(\mathfrak{M})$.
  The number of all conjugacy classes is the partition number of $n$.
\end{prop}

\begin{proof}
  Any two extension fields of a finite field $\FK$ are isomorphic if the extension degrees coincide.
  So the first assertion is a consequence of Proposition \ref{conjugatebd}.
  The second assertion follows from the first.
\end{proof}

\begin{lem}\label{extremalone}
  If $\mathfrak{M}$ is division then $\abs{\mathrm{N}(\mathfrak{M})} = n(q^n - 1)$.
\end{lem}

\begin{proof}
  We define $\rho : \FF_{q^n} \rightarrow \mathrm{End}(\FF_{q^n}/\Fq)$ by $\rho(\xi) : \eta \mapsto \xi\eta$ for all $\xi, \eta \in \FF_{q^n}$.
  Since $\mathfrak{M}$ is division and $\Fq^n \cong \FF_{q^n}$ as $\Fq$-vector spaces we may assume without loss of generality that $\mathfrak{M} = \rho(\FF_{q^n})$.
  Let $\sigma \in \mathrm{N}(\mathfrak{M})$.
  We claim that $\sigma = \rho(\omega)\varphi$ for a unique $\omega \in \FF_{q^n}$ and $\varphi \in \mathrm{Aut}(\FF_{q^n}/\Fq)$.
  Set $\omega := \sigma(1)$ and $\varphi := \rho(\omega)^{-1}\sigma = \rho(\omega^{-1})\sigma$.
  Then $\varphi(1) = 1$ and $\varphi^{-1}\mathfrak{M}\varphi = \mathfrak{M}$, i.e., $\varphi \in \mathrm{N}(\mathfrak{M})$.
  So, for each $\xi \in \FF_{q^n}$, there is $\xi' \in \FF_{q^n}$ such that $\varphi^{-1}\rho(\xi')\varphi = \rho(\xi)$, i.e., $\varphi(\xi\eta) = \xi' \varphi(\eta)$ for all $\eta \in \FF_{q^n}$.
  Thus $\xi' = \varphi(\xi)$ whence $\varphi(\xi\eta) = \varphi(\xi)\varphi(\eta)$, i.e., $\varphi \in \mathrm{Aut}(\FF_{q^n}/\Fq)$, as claimed.
  Conversely $\rho(\FF_{q^n}^{\times})\mathrm{Aut}(\FF_{q^n}/\Fq) \subseteq \mathrm{N}(\mathfrak{M})$.
  Since $\abs{\mathrm{Aut}(\FF_{q^n}/\Fq)} = n$, $\abs{\rho(\FF_{q^n}^{\times})} = q^n - 1$ and $\rho(\FF_{q^n}^{\times})\cap\mathrm{Aut}(\FF_{q^n}/\Fq) = \set{\mathrm{Id}}$ we are done.
\end{proof}

\begin{lem}\label{extremaltwo}
  If $\mathfrak{M}$ is split then $\abs{\mathrm{N}(\mathfrak{M})} = n!(q - 1)^n$.
\end{lem}

\begin{proof}
  We may assume that $\mathfrak{M} = \mathfrak{D}$.
  It is clear that $\mathfrak{D}^{\ast}\mathfrak{P} \subseteq \mathrm{N}(\mathfrak{D})$ where $\mathfrak{D}^{\ast} = \mathfrak{D} \cap \GLnq$ and $\mathfrak{P}$ is the set of all $n\times n$ permutation matrices.
  Conversely let $U = (\alpha_{i,j}) \in \mathrm{N}(\mathfrak{D})$.
  If we define $D_i := \mathrm{diag}(0, \ldots, 0, 1, 0, \ldots, 0)$ where $1$ is in the $i$th position then $D_i U = U D$ for some $D = \mathrm{diag}(\delta_1, \ldots, \delta_n)$.
  We have $(0, \ldots, 0, \alpha_{i,j}, 0, \ldots, 0)^{\mathrm{T}} = \delta_j (\alpha_{1,j}, \ldots, \alpha_{n,j})^{\mathrm{T}}$ by comparing the $j$th columns of both sides of $D_i U = U D$.
  Since $\alpha_{i,j_i} \neq 0$ for some $j_i$, we have $\delta_{j_i} = 1$ so that $\alpha_{k,j_i} = 0$ for all $k \neq i$.
  Since $U$ is invertible, $\alpha_{i,j} = 0$ for all $j \neq j_i$.
  Thus $U \in \mathfrak{D}^{\ast}\mathfrak{P}$.
  Since $\abs{\mathfrak{D}^{\ast}} = (q - 1)^n$, $\abs{\mathfrak{P}} = n!$ and $\mathfrak{D}^{\ast} \cap \mathfrak{P} = \set{I_n}$ we are done.
\end{proof}

Recall that there is another expression of a partition $(n_1 , \ldots n_t)$ of $n$, i.e., $(1^{m_1}, 2^{m_2} , \ldots, n^{m_n})$ where $m_i$ is the multiplicity of $i$ in $(n_1 , \ldots n_t)$.
For instance we see that $(4,2,2,1) = (1^1, 2^2, 3^0, 4^1, 5^0, 6^0, 7^0, 8^0, 9^0)$.

\begin{lem}\label{normalizer}
  Suppose that $\mathfrak{M}$ has type $(1^{m_1}, \ldots, n^{m_n})$.
  Then
  \[ \abs{\mathrm{N}(\mathfrak{M})} = \prod_{i = 1}^n m_i! \big(i(q^i - 1)\big)^{m_i}. \]
\end{lem}

\begin{proof}
  We may assume that $\mathfrak{M} = \bigoplus_{i=1}^t \rho_i(\FF_i)$ with $\DimK{\FF_i} = n_i$ and $(1^{m_1}, \ldots, n^{m_n}) = (n_1, \ldots, n_t)$.
  By Lemmas \ref{extremalone} and \ref{extremaltwo} it is sufficient to prove that if $U \in \mathrm{N}(\mathfrak{M})$ then $U$ is a block matrix $(B_{i,j})_{1 \leq i,j \leq t}$ with $B_{i,j} \in \mathrm{Mat}_{n_i \times n_j}(\FK)$ such that $B_{i,j}$ is the zero matrix whenever $n_i \neq n_j$.
  Suppose on the contrary that $n_i \neq n_j$ and $B_{i,j} \neq 0$.
  Since $\rho_i(\FF_i) B_{i,j} = B_{i,j} \rho_j(\FF_j)$ we know from Corollary \ref{regstdmodofmt} that there exists a nonzero map $\psi : \FF_j \rightarrow \FF_i$ such that $\forall\xi \in \FF_j$, $\exists\xi' \in \FF_i$, $\forall\eta\in\FF_j$, $ \psi(\xi\eta) = \xi' \psi(\eta)$.
  If we put $\varphi := \psi(1)^{-1}\psi$ then $\varphi(1) = 1$ and $\varphi(\xi\eta) = \xi'\varphi(\eta)$ for all $\eta \in \FF_j$.
  Since $\xi' = \varphi(\xi)$ we have $\varphi(\xi\eta) = \varphi(\xi)\varphi(\eta)$, i.e., $\varphi$ is a ring homomorphism.
  So $\FF_i$ becomes an $\FF_j$-module.
  It is impossible that $n_i < n_j$, because $\FF_j$ itself is an irreducible $\FF_j$-module.
  Since $B_{i,j}^{\mathrm{T}}\rho_i^{\mathrm{T}}(\FF_i) = \rho_j^{\mathrm{T}}(\FF_j)B_{i,j}^{\mathrm{T}}$ the case $n_i > n_j$ is impossible, too.
\end{proof}

\begin{prop}\label{conjugacysize}
  We have $\abs{\mathrm{Cl}(\mathfrak{M})} = \abs{\GLnq} / \abs{\mathrm{N}(\mathfrak{M})}$.
\end{prop}

\begin{proof}
  It follows from the orbit-stabilizer counting principle.
\end{proof}

Now we recall the important formula of A. Cayley:

\begin{prop}\label{cayleyformula}
  For every $n \in \Natural$ with indeterminate $q$ we have
  \begin{enumerate}
  \item \[\sum_{(1^{m_1}, \ldots, n^{m_n}) \vdash n} \Bigg(\prod_{i=1}^n \frac{1}{m_i! \big( i(1 - q^i)\big)^{m_i}} \Bigg) = \prod_{i=1}^n \frac{1}{1 - q^i}  ;\]
  \item \[\sum_{(1^{m_1}, \ldots, n^{m_n}) \vdash n} \Bigg(\prod_{i=1}^n \frac{1}{m_i! \big( i(q^i - 1)\big)^{m_i}} \Bigg) = \frac{q^{\frac{n(n-1)}{2}}}{\prod_{i=1}^n (q^i - 1)} .\]
  \end{enumerate}
\end{prop}

\begin{proof}
  (i)
  It is the well-known Cayley's formula \cite[Example 1, p. 209]{Andrews}.
  (ii)
  Replace $q$ in (i) by $1/q$.
  And use the fact that $n = \sum_{i=1}^n im_i$.
\end{proof}

\begin{thm}\label{numofmtgl}
  The number of all maximal tori in $\glnq$ is $q^{n(n-1)}$.
\end{thm}

\begin{proof}
  Recall that $\abs{\GLnq} = \prod_{i = 1}^n (q^n - q^{n-i}) = q^{\frac{n(n-1)}{2}} \prod_{i = 1}^n (q^i - 1)$.
  So we are done by Proposition \ref{cayleyformula} (ii), because the desired number is
  \[\sum_{\mathfrak{M}} \abs{\mathrm{Cl}(\mathfrak{M})} = \sum_{\mathfrak{M}} \frac{\abs{\GLnq}}{\abs{\mathrm{N}(\mathfrak{M})}} = \abs{\GLnq} \times \sum_{\mathfrak{M}} \frac{1}{\abs{\mathrm{N}(\mathfrak{M})}} \]
  where the sum is over maximal tori of distinct types.
\end{proof}

\begin{rem}
  Accidentally, the quantity $q^{n(n-1)}$ is also well-known as the number of all nilpotent matrices in $\mathrm{Mat}(n,q)$.
  But I do not know whether or not there is a bijection between the set of all maximal tori in $\glnq$ and the set of all nipotent matrices in $\mathrm{Mat}(n,q)$.
\end{rem}

\section{Maximal tori in $\slnK$}

Throughout this section we also use the symbol $\FK$ as a field of positive characteristic $p$.
We denote respectively by $\mathrm{MT}_{\glnK}$ and $\mathrm{MT}_{\slnK}$ the sets of all maximal tori in $\glnK$ and $\slnK$.

We first review very well-known and basic facts in Lie theory.
(The proof of Proposition \ref{verybasicfact} (i) comes from \cite{Jacobson2}.)

\begin{prop}\label{verybasicfact}
  Suppose $p > 2$ or $n >2$.
  Then the following hold:
  \begin{enumerate}
  \item The only nontrivial ideals of $\glnK$ are $\FK I_n$ and $\slnK$;
  \item If $p \ndv n$ then there exists no nontrivial ideal of $\slnK$.
    If $p \dv n$ then $\FK I_n$ is the only nontrivial ideal of $\slnK$.
  \end{enumerate}
\end{prop}

\begin{proof}
  (i)
  Clearly, $\FK I_n$ and $\slnK$ are ideals of $\glnK$.
  Suppose that $\mathfrak{J}$ is a nonzero ideal$\neq \FK I_n$ of $\glnK$.
  Let $X = \sum_{i,j} \alpha_{i,j} E_{i,j} \in \mathfrak{J}\setminus \FK I_n$ where $\alpha_{i,j} \in \FK$ and $E_{i,j}$'s are elements of the standard basis for $\mathrm{Mat}(n,\FK)$.
  Assume $\alpha_{s,t} \neq 0$ for some $s,t$ with $s \neq t$.
  If $p > 2$ then $-2\alpha_{s,t}E_{t,s} = [[X,E_{t,s}],E_{t,s}] \in \mathfrak{J}$.
  So $E_{t,s} \in \mathfrak{J}$.
  If $n > 2$ then choose $r \neq s,t$.
  Then $\alpha_{s,t}E_{r,s} = [[[X,E_{t,s}],E_{r,s}],E_{r,r}] \in \mathfrak{J}$.
  So $E_{r,s} \in \mathfrak{J}$.
  On the other hand, if $\alpha_{s,t} = 0$ for all $s,t$ with $s \neq t$ then $X = \sum_i \alpha_{i,i}E_{i,i}$ so that $\alpha_{s,s} \neq \alpha_{t,t}$ for some $s,t$ with $s \neq t$.
  Since $(\alpha_{s,s} - \alpha_{t,t})E_{s,t} = [X, E_{s,t}] \in \mathfrak{J}$ we have $E_{s,t} \in \mathfrak{J}$.
  In all cases, $E_{s,t} \in \mathfrak{J}$ for some $s,t$ with $s \neq t$.
  Thus, $E_{i,t} = [E_{i,s}, E_{s,t}] \in \mathfrak{J}$ for all $i \neq t$, and $E_{s,j} = [E_{s,t}, E_{t,j}] \in \mathfrak{J}$ for all $j \neq s$.
  This facts imply that,  for all $i,j$ with $i \neq j$, $E_{i,j} \in \mathfrak{J}$ and so $E_{i,i} - E_{j,j} = [E_{i,j}, E_{j,i}] \in \mathfrak{J}$.
  Since $\mathfrak{J}$ is a proper ideal of $\glnK$ we have $\mathfrak{J} = \slnK$.
  (ii)
  Note that $p \dv n$ if and only if $I_n \in \slnK$.
  If $p \dv n$ then $\FK I_n$ is an ideal of $\slnK$.
  Let $\mathfrak{J}$ be a nonzero ideal$\neq \FK I_n$  of $\slnK$.
  The proof is similar with (i) except the case $n > 2$, since $E_{r,r} \not\in \slnK$.
  Thus suppose $p = 2$ and $n > 2$.
  Then $\alpha_{s,t}E_{r,s} + \alpha_{s,r}E_{t,s} = [[X, E_{t,s}], E_{r,s}] \in \mathfrak{J}$.
  If $\alpha_{s,r} = 0$ then $E_{r,s} \in \mathfrak{J}$.
  If $\alpha_{s,r} \neq 0$ then $\alpha_{s,r}E_{t,s} = [[[X, E_{t,s}], E_{r,s}], E_{r,r} + E_{s,s}] \in \mathfrak{J}$.
  So $E_{t,s} \in \mathfrak{J}$.
\end{proof}

\begin{rem}
  Two ideals $\FK I_n$ and $\slnK$ in $\glnK$ are in fact closed under the $p$-map, i.e., $\FK I_n$ and $\slnK$ are \textit{restricted} ideals of $\glnK$.
\end{rem}

Recall that if $p \ndv n$ then $\glnK$ is a direct sum of two restricted ideals $\FK I_n$ and $\slnK$ in $\glnK$, that is, $\glnK = \FK I_n\oplus \slnK$ as vector spaces and $[X,Y] = 0$ for every $X \in \FK I_n$ and $Y \in \slnK$.

\begin{lem}\label{bijglslpndvn}
  Let $p \ndv n$ and $\pi : \glnK \rightarrow \glnK / \FK I_n \cong \slnK$ a canonical projection.
  Then $\pi$ induces a bijection $\tilde{\pi}: \mathrm{MT}_{\glnK} \rightarrow \mathrm{MT}_{\slnK}$ described by $\tilde{\pi}(\mathfrak{M}) = \mathfrak{M} \cap \slnK$ and $\tilde{\pi}^{-1}(\mathfrak{M}_0) = \FK I_n \oplus \mathfrak{M}_0$.
\end{lem}

\begin{proof}
  If we denote the canonical isomorphism $\glnK / \FK I_n \rightarrow \slnK$ by $\varphi$ then it is easy to see that $\tilde{\pi}(\mathfrak{M}) := (\varphi\circ\pi)(\mathfrak{M}) = \mathfrak{M} \cap \slnK$ for each $\mathfrak{M} \in \mathrm{MT}_{\glnK}$.
  We know from Proposition \ref{maximaldim} that, for any $\mathfrak{M}_0 \in \mathrm{MT}_{\slnK}$, there exists some $\mathfrak{M} \in \mathrm{MT}_{\glnK}$ with $\mathfrak{M}_0 \subseteq \mathfrak{M}$.
  So $\tilde{\pi}$ is surjective.
  Since $p \ndv n$, $I_n \not\in \mathfrak{M}_0$ for all $\mathfrak{M}_0 \in \mathrm{MT}_{\slnK}$.
  Since $I_n \in \mathfrak{M}$ for each $\mathfrak{M} \in \mathrm{MT}_{\glnK}$ by Proposition \ref{closedmp} (ii), $\tilde{\pi}$ is injective.
\end{proof}

\begin{lem}\label{bijglslpdvn}
  Suppose $p > 2$ or $n > 2$.
  If $p \dv n$ and we define $\psi : \mathrm{MT}_{\glnK} \rightarrow \mathrm{MT}_{\slnK}$ by $\mathfrak{M} \mapsto \mathfrak{M} \cap \slnK$ then $\psi$ is a bijection.
\end{lem}

\begin{proof}
  We define $\eta : \mathrm{MT}_{\slnK} \rightarrow \mathrm{MT}_{\glnK}$ by $\eta(\mathfrak{M}_0) := \gn{\mathfrak{M}_0}$ where $\gn{\mathfrak{M}_0}$ is the associative $\FK$-algebra in $\mathrm{Mat}(n,\FK)$ generated by $\mathfrak{M}_0$.
  If we prove $\eta$ to be well-defined then the proof will be done, because $\eta$ is obviously the inverse of $\psi$.
  So we claim $\Dim{\gn{\mathfrak{M}_0}} = n$.
  Assume to the contrary that $\gn{\mathfrak{M}_0} = \mathfrak{M}_0$.
  Then $\mathfrak{M}_0 = \bigoplus_{i=1}^s \FE_i$ for some separable extension fields $\FE_1, \ldots, \FE_s$ of $\FK$ as in the proof of Corollary \ref{dsofefs}.
  We put $r_i := \mathrm{rank}(1_{\FE_i})$, $d_i := [\FE_i : \FK]$ and $m_i := r_i / d_i$ for each $i \in \{1, \ldots, s\}$.
  Note that $m_i$'s are integers and, for all $X = \sum_{i} X_i \in \mathfrak{M}_0$ with $X_i \in \FE_i$,
  \[\mathrm{Tr}(X) = \sum_{i=1}^s m_i \mathrm{Tr}_{\FE_i / \FK}(X_i).\]
  Since $\sum_i d_i = n - 1$ and $\sum_i m_i d_i = \sum_i r_i = n$ we can see that $m_j = 2$ and $d_j = 1$ for some $j \in \{1, \ldots, s\}$, and $m_i = 1$ for each $i \neq j$.
  Suppose $p > 2$.
  Since $p \ndv m_j = 2$, $m_j \mathrm{Tr}_{\FE_j / \FK}$ is not identically zero.
  So $\mathrm{Tr}(\mathfrak{M}) \neq 0$, a contradiction.
  Suppose $n > 2$.
  Then $s \geq 2$.
  Since $m_i = 1$ for $i \neq j$, $m_i \mathrm{Tr}_{\FE_i / \FK}$ is not identically zero.
  So $\mathrm{Tr}(\mathfrak{M}) \neq 0$, a contradiction.
\end{proof}

\begin{rem}
  Suppose $p = n = 2$.
  Then the unique maximal torus in $\mathfrak{sl}(2, \FK)$ is $\FK I_2$.
  Therefore the map $\psi$ is clearly not injective.
\end{rem}

\begin{cor}\label{numofmtsl}
  Suppose $p > 2$ or $n > 2$.
  For each $q = \Pl$ with $\ell \in \Natural^+$, the number of all maximal tori in $\slnq$ is $q^{n(n-1)}$.
\end{cor}

\begin{proof}
  This result is due to Lemmas \ref{bijglslpndvn}, \ref{bijglslpdvn} and Theorem \ref{numofmtgl}.
\end{proof}

\begin{cor}\label{slselfcentral}
  Suppose $p > 2$ or $n > 2$.
  Then every maximal torus in $\slnK$ is self-centralizing.
\end{cor}

\begin{proof}
  Let $\mathfrak{M}_0 \in \mathrm{MT}_{\slnK}$ and assume that $\mathfrak{M}_0 \subsetneq \mathrm{C}_{\slnK}(\mathfrak{M}_0)$.
  Choose $X \in \mathrm{C}_{\slnK}(\mathfrak{M}_0) \setminus \mathfrak{M}_0$. 
  We consider two cases:
  (a) $p \ndv n$ and (b) $p \dv n$.
  (a)
  If $p \ndv n$ then $X$ centralizes $\mathfrak{M} := \FK I_n \oplus \mathfrak{M}_0$ which is in $\mathrm{MT}_{\glnK}$ by Lemma \ref{bijglslpndvn}.
  (b)
  If $p \dv n$ then $X$ centralizes $\mathfrak{M} := \gn{\mathfrak{M}_0}$ which is in $\mathrm{MT}_{\glnK}$ by Lemma \ref{bijglslpdvn}.
  In any case, $X \in \mathfrak{M}$ by Proposition~\ref{mtingl}.
  Since $X \in \slnK$ we have $X \in \mathfrak{M} \cap \slnK = \mathfrak{M}_0$, a contradiction.
\end{proof}

\end{document}